\RequirePackage{silence}
\documentclass[a4paper,11pt]{amsart}
\usepackage[hmarginratio={1:1},vmarginratio={1:1},lmargin=60.0pt,tmargin=60.0pt]{geometry}

\allowdisplaybreaks

\usepackage[numbers]{natbib}
\usepackage[utf8]{inputenc}

\usepackage{latexsym,exscale,mathtools,textcomp}
\usepackage{amssymb,amsmath,amsthm,amsfonts,mathrsfs,bbm,enumitem,stmaryrd}
\usepackage[table]{xcolor}
\usepackage{graphicx}
\usepackage{booktabs}
\usepackage{xparse}

\usepackage{tikz}
\usetikzlibrary{arrows.meta,positioning}

\setlist[enumerate]{itemsep=0.15cm,label=\emph{\upshape(\alph*)}}
\setlist[enumerate,2]{itemsep=0.15cm,label=\emph{\upshape(\roman*)}}

\definecolor{mygray}{gray}{0.6}
\definecolor{mygraydark}{gray}{0.4}
\definecolor{mygraylight}{gray}{0.85}
\definecolor{spinach}{RGB}{46,139,87}
\definecolor{tomato}{RGB}{255,99,71}
\definecolor{orchid}{RGB}{143,40,194}
\definecolor{neon}{RGB}{77,77,255}
\definecolor{pumpkin}{RGB}{224,180,80}
\definecolor{citron}{RGB}{190,180,90}

\definecolor{lava}{RGB}{207,16,32}
\definecolor{cream}{RGB}{255,253,208}
\definecolor{verdigris}{RGB}{67,179,174}
\definecolor{Black}{RGB}{0,0,0}
\definecolor{mydarkblue}{RGB}{10,10,170}
\definecolor{darkspinach}{RGB}{20,70,20}
\definecolor{darktomato}{RGB}{155,40,30}
\definecolor{darkorchid}{RGB}{50,10,100}
\definecolor{darklava}{RGB}{150,8,16}

\DeclarePairedDelimiterX{\set}[1]{\{}{\}}{\setargs{#1}}
\NewDocumentCommand{\setargs}{>{\SplitArgument{1}{|}}m}{\setargsaux#1}
\NewDocumentCommand{\setargsaux}{mm}
{\IfNoValueTF{#2}{#1} {#1\,\delimsize|\,\mathopen{}#2}}

\newcommand{\C}{\mathbb{C}}

\newcommand{\Q}{\mathbb{Q}}
\newcommand{\Z}{\mathbb{Z}}
\newcommand{\F}{\mathbb{F}}

\newcommand{\Sn}{\mathfrak{S}}

\newcommand{\Parity}{\operatorname{Parity}}
\newcommand{\rank}{\operatorname{rank}}

\newcommand{\Inv}{\operatorname{Inv}}

\usepackage{aliascnt,etoolbox}
\def\NewTheorem#1{%
\newaliascnt{#1}{equation}%
\newtheorem{#1}[#1]{#1}%
\aliascntresetthe{#1}%
\expandafter\def\csname #1autorefname\endcsname{#1}%
}
\def\equationautorefname~#1\null{(#1)\null}

\numberwithin{equation}{subsection}

\NewTheorem{Proposition}
\NewTheorem{Theorem}
\NewTheorem{Corollary}
\AtEndEnvironment{Corollary}{\null\hfill$\square$}%
\NewTheorem{Lemma}
\theoremstyle{definition}
\NewTheorem{Definition}
\AtEndEnvironment{Definition}{\null\hfill$\Diamond$}%
\NewTheorem{Notation}
\AtEndEnvironment{Notation}{\null\hfill$\Diamond$}%
\NewTheorem{Example}
\AtEndEnvironment{Example}{\null\hfill$\Diamond$}%
\theoremstyle{remark}
\NewTheorem{Remark}
\AtEndEnvironment{Remark}{\null\hfill$\Diamond$}%

\usepackage[hypertexnames=false]{hyperref}
\usepackage{bookmark}
\hypersetup{
pdftoolbar=true,
pdfmenubar=true,
pdffitwindow=false,
pdfstartview={FitH},
pdftitle={Torsion proliferation},
pdfauthor={Joseph Baine and Daniel Tubbenhauer},
pdfsubject={},
pdfcreator={Joseph Baine and Daniel Tubbenhauer},
pdfproducer={Joseph Baine and Daniel Tubbenhauer},
pdfkeywords={p-canonical basis, p-Kazhdan--Lusztig basis, Kazhdan--Lusztig polynomials, Schubert varieties, interval pattern embeddings},
pdfnewwindow=true,
colorlinks=true,
linkcolor=mydarkblue,
citecolor=teal,
filecolor=magenta,
urlcolor=orchid,
linkbordercolor=lava,
citebordercolor=teal,
urlbordercolor=orchid,
linktocpage=true
}

\def\makeautorefname#1#2{\csdef{#1autorefname}{#2}}

\makeautorefname{section}{Section}%
\makeautorefname{subsection}{Section}%
\makeautorefname{subsubsection}{Section}%

\begin{document}
\title[Torsion proliferation]{Torsion proliferation}
\author[J. Baine and D. Tubbenhauer]{Joseph Baine and Daniel Tubbenhauer}

\address{J.B.: Max-Planck-Institut f\"{u}r Mathematik, Vivatsgasse 7, 53111 Bonn, Germany.}
\email{baine@mpim-bonn.mpg.de}

\address{D.T.: The University of Sydney, School of Mathematics and Statistics F07, Office Carslaw 827, NSW 2006, Australia, \href{http://www.dtubbenhauer.com}{www.dtubbenhauer.com}, \href{https://orcid.org/0000-0001-7265-5047}{ORCID 0000-0001-7265-5047}}
\email{daniel.tubbenhauer@sydney.edu.au}

\begin{abstract}
	This paper considers the generic behaviour of sheaves supported on Schubert varieties in finite flag varieties. 
	Our main results are that, for any fixed prime $p$, as the rank of the associated algebraic group grows, almost every intersection cohomology sheaf admits $p$-torsion in its stalks or costalks, and almost every indecomposable parity sheaf is not perverse. 
	For certain families of partial flag varieties we also determine explicit asymptotics for this growth. 
\end{abstract}


\addtocontents{toc}{\protect\setcounter{tocdepth}{1}}

\maketitle

%\tableofcontents

\section{Introduction}\label{S:Statement}

	Fundamental questions in modular representation theory, like determining the characters of simple modules for reductive algebraic groups (in positive characteristic) and finite groups of Lie type (in defining characteristic), are equivalent to computing the stalks of intersection cohomology sheaves on flag varieties. 
	When these stalks are devoid of $p$-torsion, their graded ranks are governed by the combinatorics of Kazhdan--Lusztig polynomials. 
	However, when $p$-torsion is present, their graded ranks are governed by the considerably more subtle $p$-Kazhdan--Lusztig polynomials. 
	\\
	\par  
	Historically, it was understood that $p$-torsion existed in the stalks or costalks (henceforth we just write `(co)stalks') of intersection cohomology sheaves on certain small-rank flag varieties. 
	Thus, it was natural to ask two questions: 
	(1) Which intersection cohomology sheaves admit $p$-torsion in their (co)stalks; and 
	(2) For which primes $p$ does $p$-torsion exist?  
	The latter question appeared in the guise of Lusztig's conjecture and James' conjecture. 
	Williamson's torsion explosion famously showed that the primes $p$ for which $p$-torsion exists grow at least exponentially in the rank of the group; see \cite{Williamson-explosion}. 
	However, there are few intersection cohomology sheaves for which we can explicitly say $p$-torsion exists in their (co)stalks, and their generic behaviour is far from understood. 
	\\
	\par 
	The results of this paper concern the first question.
	Notation will be fixed in \autoref{S:Background}. 
	In the meantime, we write $W_n$ for the Weyl group of a rank $n$ classical algebraic group $G$ with Borel subgroup $B$, and  $\mathcal{IC}_x^{\Z}$ for the integral intersection cohomology sheaf supported on the Schubert variety $\overline{B x B / B}$ in the flag variety $G/B$. 
	The corresponding Kazhdan--Lusztig basis and $p$-Kazhdan--Lusztig basis elements are denoted $b_x$ and ${}^p b_x$ respectively. 
	The first main result of this paper is:  
	
	\begin{Theorem}\label{T:Main}
		Fix a prime $p$ and a tower $\{ G_n\}$ of classical algebraic groups with Weyl groups $\{W_n\}$. Then 
		\begin{align*}
			\lim_{n \rightarrow \infty}
			\frac{|\{ x \in W_n  ~|~ p\text{-torsion exists in the (co)stalks of } \mathcal{IC}_x^{\Z} \}|}{|W_n|}
			=
			1.
		\end{align*}
		Equivalently
		\begin{align*}
            \lim_{n \rightarrow \infty}
            \frac{|\{ x \in W_n ~|~  {}^p b_x \neq b_x \}|}{|W_n|} = 1.
        \end{align*}
        Moreover, the complementary density decays at least exponentially in $n$.
	\end{Theorem}
	
	In other words, eventually almost every intersection cohomology sheaf on a full flag variety admits $p$-torsion in its (co)stalks. 
	This result is in stark contrast with the current state of knowledge. 
	For example, the presence of $p$-torsion in the (co)stalks of intersection cohomology sheaves on full flag varieties has only been determined in low ranks (typically $n \leq 7$), where it exists only for small primes ($p \in \{2,3\}$) and is typically rare ($2$-torsion is present in $38$ of the $40320$ possible intersection cohomology sheaves for $SL_8$); see \cite{WB12}. 
	Furthermore, there are several families of partial flag varieties where it is known which intersection cohomology sheaves admit $p$-torsion in their (co)stalks; see \cite{BowmanDeVisscherHaziNorton-hermitian, Baine-cominuscule}. 
	Interestingly, we show, in \autoref{S:Partial}, that the analogue of \autoref{T:Main} in the parabolic setting admits various limits depending on the partial flag variety.  
	\\ 
	\par 
	The question of whether an intersection cohomology sheaf $\mathcal{IC}_x^{\Z}$ has $p$-torsion in its (co)stalks is equivalent to asking whether the corresponding parity sheaf $\mathcal{E}_x$ (with coefficients in a field $\Bbbk$ of characteristic $p$) is an intersection cohomology sheaf. 
	A less restrictive question one could ask is whether the parity sheaf $\mathcal{E}_x$ is perverse.  
	Our second, and stronger, main result states that for a fixed $p$, this is asymptotically almost never the case. 
	
	\begin{Theorem}\label{T:Main2}
		Fix a field $\Bbbk$ of characteristic $p>0$ or a complete local ring with residue field of characteristic $p$,  and a tower $\{ G_n\}$ of classical algebraic groups with Weyl groups $\{W_n\}$. 
		Then
		\begin{align*}
			\lim_{n \rightarrow \infty}
			\frac{|\{ x \in W_n ~|~ \mathcal{E}_x \text{ is not perverse}\}|}{|W_n|}
			=
			1.
		\end{align*}
		Moreover, the complementary density decays at least exponentially in $n$.
	\end{Theorem}
	
	Again, this theorem is in stark contrast with the state of knowledge. 
	Very few non-perverse parity sheaves are known on full flag varieties \cite{LW17,McN18}.
	It follows from \cite{Baine-cominuscule} that the Whittaker parity sheaves which are Koszul dual to mixed tilting sheaves on minuscule and cominuscule partial flag varieties are perverse for all primes $p$.
	Moreover, parity sheaves in the category of $G(\C[\![t]\!])$-equivariant sheaves on the affine Grassmannian are perverse for $p$ sufficiently large; see \cite{JMW-tilting}. 
	\\
	\par 	 
    The proofs of these theorems utilise geometric and combinatorial methods.
    We begin by considering indecomposable parity sheaves (with coefficients in $\Z_p$) supported on Schubert varieties in finite flag varieties. 
    When we extend scalars from $\Z_p$ to $\Q_p$ these parity sheaves can decompose and their graded decomposition numbers encode whether $\mathcal{IC}_x^{\Z}$ has $p$-torsion in its (co)stalks, and whether $\mathcal{E}_x$ is not perverse. 
    The former occurs when the parity sheaf admits a non-trivial decomposition; the latter when the graded decomposition numbers are not integers. 
    These non-trivial decompositions are associated to certain singular Schubert varieties. 
    Using the notion of interval pattern embeddings, we explain how to embed these singularities into Schubert varieties in flag varieties for classical groups of larger rank. 
    Finally, using combinatorial techniques, we show that these singularities embed into almost every Schubert variety for a sufficiently large classical algebraic group. 
	\\

\noindent\textbf{Acknowledgements.}
We are grateful to Kevin Coulembier and Jensen O'Sullivan for helpful conversations, useful comments, and encouragement.
This paper is part of ARC Discovery Project DP250100762. 
DT acknowledges support from ARC Future Fellowship FT230100489, and notes that torsion seems to proliferate more reliably than most desirable phenomena.
 \\
\noindent\textbf{AI declaration:} AI was not used for the proof strategy or writing, but ChatGPT-6.0 and Claude Opus 5 were used for the final proofreading. 

\section{The groundwork}\label{S:Background}

	We briefly fix notation regarding various intersection cohomology sheaves, parity sheaves and $p$-Kazhdan--Lusztig bases. 
	\\
	\par 
    Let $G$ be a complex, connected classical algebraic group, with a fixed choice of Borel subgroup $B \subset G$ and maximal torus $T \subset B$.
    Associated to $G \supset B \supset T$ are the Weyl group $W$ and its simple reflections $S$.  
    When we speak of a tower of classical algebraic groups $\{G_n\}$, we mean a sequence of classical algebraic groups $\{G_n\}$, where $G_n$ has rank $n$, each $G_n$ is of the same type (e.g. $SL_{n+1}$, $Sp_{2n}$, $SO_{2n+1}$, $SO_{2n}$), and a sequence of embeddings that embed maximal tori into maximal tori and Borel subgroups into Borel subgroups. 
    If $W_n$ denotes the Weyl group of $G_n$, then the embeddings induce inclusions $W_{n-1} \hookrightarrow W_n$ that are compatible with the Coxeter structure. 
    Denote the length function by $\ell$ and the Bruhat order by $\leq$.
    When $n$ is clear from context or unimportant, we omit it from notation.  
    \\
    \par 
    The flag variety $G/B$ admits a stratification by $B$-orbits whose strata, the Schubert cells $BxB/B$, are indexed by elements $x \in W$.
    Write $i_x : BxB/B \hookrightarrow G/B$ for the inclusion.   
    \\
    \par  
    Let $D_{(B)}^b(G/B , \Z)$ denote the bounded derived category of sheaves of $\Z$-modules on $G/B$ that are constructible with respect to the Schubert stratification. 
    Its triangulated shift functor is denoted by $[1]$.
    The intersection cohomology sheaf (for the middle perversity) supported on $\overline{BxB/B}$ is $\mathcal{IC}_x^{\Z} := i_{x !*} \underline{\Z}_{BxB/B} [\ell(x)]$.  
    \\
    \par 
	Consider the $p$-modular system $(\Q_p , \Z_p , \F_p)$. 
	Define the simple intersection cohomology sheaf $\mathcal{IC}_x$ as $i_{x !*} \underline{\F_p}_{BxB/B} [\ell(x)]$.
	Let $\Parity(G/B , \Z_p)$ be the full subcategory of $D_{(B)}^b(G/B , \Z_p)$ of parity sheaves in the sense of \cite{JMW-parity}. 
	Up to isomorphism, there is a unique indecomposable, self-dual parity sheaf $\mathcal{E}_x$ whose support is $\overline{BxB/B}$.
    Let $i_y^* \mathcal{E}_x$ denote the stalk of $\mathcal{E}_x$ at $yB/B$. 
    The $p$-Kazhdan--Lusztig polynomial ${}^p h_{y,x}$ is defined as the Poincar\'{e} polynomial of the stalk $i_y^* \mathcal{E}_x$. 
    Namely
    \begin{align}\label{E:pKL}
    	{}^p h_{y,x} 
    	:= 
    	\sum_{j \in \Z } \rank H^j (i_y^* \mathcal{E}_x) v^{-\ell(y)-j}.
    \end{align}
    Denote by $\mathcal{E}_x^{\Q_p}$ and $\mathcal{E}_x^{\F_p}$ the corresponding indecomposable parity sheaves with coefficients in $\Q_p$ and $\F_p$ respectively. 
    Extending scalars from $\Z_p$ to $\Q_p$ produces a decomposition\footnote{Note that the decomposition is not canonical, but the multiplicities ${}^pa_{y,x,i}$ are independent of the decomposition. }
    \begin{align}\label{E:decomp}
    	\mathcal{E}_x \otimes_{\Z_p}^L \Q_p \cong \bigoplus_{y \in W, \, i \in \Z} (\mathcal{E}_y^{\Q_p} [i])^{\oplus {}^p a_{y,x,i}}
    \end{align}
    for some non-negative integers ${}^pa_{y,x,i}$ satisfying ${}^pa_{y,x,i} = 0$ for $y \not\leq x$, ${}^pa_{x,x,i}=0$ if $i \neq 0$ and ${}^pa_{x,x,i}=1$ if $i=0$. 
    Note that the self-duality of $\mathcal{E}_x$ implies ${}^p a_{y,x,i} = {}^p a_{y,x,-i}$ for all $i \in \Z$. 
    For more details see \cite[\S2.5]{JMW-parity}.  
    \\
	\par 
	The Hecke algebra $H$ associated to the Coxeter system $(W,S)$ is a $\Z[v,v^{-1}]$-algebra with standard basis $\{ \delta_x \,|\, x \in W \}$, see \cite{Tits69}, and Kazhdan--Lusztig basis $\{ b_x \,|\, x \in W \}$, see \cite{KL79}. 
	The regular representation of the Hecke algebra is isomorphic to the split Grothendieck group of $\Parity(G/B , \Z_p)$. 
	In particular
	\begin{align*}
		[\Parity(G/B , \Z_p)] 
		\tilde{\longrightarrow} 
		H, 
		&&
		[\mathcal{E}_x] 
		\longmapsto 
		{}^p b_x := \sum_{y \in W} {}^p h_{y,x} \, \delta_y,
	\end{align*}
	where $[\mathcal{E}[1]] = v[\mathcal{E}]$.
	The elements $\{ {}^p b_x \,|\, x \in W \}$ form the $p$-Kazhdan--Lusztig basis of $H$. 
	Note that the $p$-Kazhdan--Lusztig basis depends on the root datum of $G$, rather than just the Coxeter system $(W,S)$.
	Importantly, the decomposition in \autoref{E:decomp} implies: 
	\begin{align}\label{E:p-can in can}
		{}^p b_x = \sum_{y \in W} {}^pa_{y,x} \, b_y 
		~~~\text{ where }~~~
		{}^p a_{y,x} \in \Z[v+v^{-1}]
	\end{align}
	and we have ${}^pa_{x,x} = 1$ and ${}^pa_{y,x} =0$ for $y \not\leq x$.  
	In particular, ${}^pa_{y,x} = \sum_{i \in \Z} {}^pa_{y,x,i}~ v^i$. 
	\\
	\par 
 	The $p$-Kazhdan--Lusztig basis encodes the following information about $\mathcal{IC}_x$. 
 	
 	\begin{Lemma}\label{L: torsion equivalences}
 		The following are equivalent: 
 		\begin{enumerate}
 			\item ${}^p b_x = b_x$;
 			\item ${}^pa_{y,x} =0$ for all $y \neq x$; 
 			\item $\mathcal{E}_x^{\F_p} \cong \mathcal{IC}_x$;
 			\item the stalks and costalks of $\mathcal{IC}_x^{\Z}$ are free of $p$-torsion. 
 		\end{enumerate}
 	\end{Lemma}
 	\begin{proof}
 		The equivalence of (a) and (b) is clear. 
 		By \cite[\S2.5]{JMW-parity} we have  $\mathcal{E}_x \otimes_{\Z_p}^L \F_p \cong \mathcal{E}_x^{\F_p}$. 
 		The equivalence of (a), (c) and (d) then follows from the previous isomorphism and \cite[\S3]{WB12}. 
 	\end{proof}
 	
 	We will also use the following numerical condition for determining when a parity sheaf is perverse.
 	
 	\begin{Lemma}\label{L: numerical non-perversity}
 		The indecomposable parity sheaf $\mathcal{E}_x$ is perverse if and only if ${}^p a_{y,x} \in \Z$ for all $y$. 
 	\end{Lemma}

     \begin{proof}
 		Recall that a sheaf $\mathcal{F}$ in $D_{(B)}^b(G/B , \Z_p)$ is perverse if and only if for all $w \in W$ we have $H^j(i_w^* \mathcal{F}) = 0$ for all $j > -\ell(w)$ and $H^j(i_w^! \mathcal{F}) = 0$ for all $j < -\ell(w)$. 
 		For $\mathcal{E}_x$ the stalk condition is equivalent to ${}^p h_{w,x} \in \Z[v]$, by \autoref{E:pKL}. 
 		The costalk condition is also equivalent to ${}^p h_{w,x} \in \Z[v]$ by Verdier duality and the self-duality of $\mathcal{E}_x$. 
 		Hence $\mathcal{E}_x$ is perverse if and only if ${}^p h_{w,x} \in \Z[v]$ for all $w \in W$. 
 		Let $h_{y,w}$ denote the classical Kazhdan--Lusztig polynomial, i.e. the polynomials defined by the equality $b_w = \sum_y h_{y,w} \, \delta_y $. 
        Equation \autoref{E:p-can in can} implies ${}^p h_{y,x} = \sum_w  {}^p a_{w,x} \, h_{y,w}$.
 		The classical Kazhdan--Lusztig polynomials satisfy the following properties $h_{y,w} = 0$ when $y \not \leq w$, $h_{w,w} =1$, and $h_{y,w} \in \Z[v]$.
        In particular, if ${}^p a_{y,x} \in \Z$ for all $y$, then ${}^p h_{w,x} \in \Z[v]$ for all $w$. 
        Conversely, if ${}^p h_{w,x} \in \Z[v]$ for all $w$ then downward induction on $y$, using the unitriangularity of Kazhdan--Lusztig polynomials, shows ${}^p a_{y,x} \in \Z[v]$.
        But ${}^p a_{y,x} \in \Z[v+v^{-1}]$, so ${}^p a_{y,x} \in \Z$ for all $y$. 
 	\end{proof}

\section{The seeds}\label{S:Seed}

    To apply our method of constructing infinitely many intersection cohomology sheaves whose (co)stalks contain $p$-torsion, or infinitely many non-perverse parity sheaves, we first require the existence of a single sheaf with either of these properties. 
    To that end, we note the following two lemmata.  
    Throughout $\Sn_m$ denotes the symmetric group on $m$ letters. 
    
	\begin{Lemma}\label{L:Seed Torsion}
		For every prime $p$, there exists a minimal integer $T(p)$ and an element $w \in \Sn_{T(p)}$ such that the (co)stalks of $\mathcal{IC}_w^{\Z}$ contain $p$-torsion. 
		Moreover, one has $T(p) \leq 4p$. 
	\end{Lemma}
	
	\begin{proof}
		This follows from Polo's generalised Kashiwara--Saito singularity as explained in \cite[\S1.5]{Williamson-hecke-category}, together with \autoref{L: torsion equivalences}. 
		In particular, the variety $Y$ of linear maps $A_1, A_2 ,B_1, B_2 : \C^p \rightarrow \C^p$ satisfying
		\begin{align*}
			B_i A_i =0 ~\text{with}~ i \in \{1,2\},
			&&
			\rank \begin{pmatrix} A_2 \\ A_1 \end{pmatrix} \leq 1,
			&&
			\rank \begin{pmatrix} B_1 ~ B_2 \end{pmatrix} \leq 1,
		\end{align*}
		is smoothly equivalent to a singularity in a Schubert variety in the flag variety of $SL_{4p}$, which gives rise to $p$-torsion in the (co)stalks of the corresponding intersection cohomology sheaf. 
	\end{proof}
	
	\begin{Lemma}\label{L:Seed Perverse}
		For every prime $p$, there exists a minimal integer $P(p)$ and an element $w \in \Sn_{P(p)}$ such that $\mathcal{E}_w$ is not perverse. 
		Moreover, one has $P(2) \leq 15$, and $P(p) \leq 5p$ when $p \neq 2$.
	\end{Lemma}
	
	\begin{proof}
		When $p=2$ this is the main result of \cite{LW17}.
		When $p>2$ this follows from McNamara's further generalisation of the Kashiwara--Saito singularity as explained in \cite[\S 4]{McN18}.
		We briefly recall the details.
		Let $J$ denote the $p \times p$ anti-diagonal matrix. 
		McNamara shows (in the special case where $d=1$ and $l=3$ in McNamara's notation) that the variety $Y$ of linear maps $A_1, A_2, A_3, B_1, B_2, B_3 : \C^p \rightarrow \C^p$ satisfying
		\begin{align*}
			B_i J A_i =0 ~\text{with}~ i \in \{1,2,3\},
			&&
			\rank \begin{pmatrix} A_3 \\ A_2 \\ A_1 \end{pmatrix} \leq 1,
			&&
			\rank \begin{pmatrix} B_1 ~ B_2 ~ B_3\end{pmatrix} \leq 1,
		\end{align*}
		arises as a slice to a Schubert variety associated to an explicit permutation $w \in \Sn_{5p}$. 
		By showing that a canonical, indecomposable direct summand of the pushforward of the constant sheaf (with coefficients in $\Z_p$) along any resolution of singularities of $Y$ is not perverse, McNamara deduces that the parity sheaf $\mathcal{E}_w$ is not perverse. 
	\end{proof}
	
	Suppose there is $p$-torsion in the (co)stalks of $\mathcal{IC}_w^{\Z}$ (resp. $\mathcal{E}_w$ is not perverse). 
	Then, by \autoref{L: torsion equivalences} (resp. \autoref{L: numerical non-perversity}), it follows that ${}^p b_w \neq b_w$ and there exists a permutation $u<w$ such that ${}^p a_{u,w} \neq 0$ (resp. ${}^p a_{u,w} \notin \Z$). 
	We call the Bruhat interval $[u,w]$ a $p$-torsion seed (resp. non-perverse seed).
	More generally, either a $p$-torsion seed or a non-perverse seed will be referred to as a seed. 

	\begin{Example}
		When $p=3$ in \autoref{L:Seed Perverse} the relevant permutations, in one-line notation, are  
		\begin{align*}
			w= c~3~2~f~b~9~a~8~6~7~5~1~e~d~4,
			&&
			u= 3~2~1~c~b~a~9~8~7~6~5~4~f~e~d
		\end{align*}
		where $a=10$, $b=11$, \dots, $f=15$. 
	\end{Example}

\section{Propagating the seeds}\label{S:Embedding}

	In this section we recall the notion of a generalised interval pattern embedding, as introduced in \cite{Woo-arbitrary}. 
	We then show how to produce interval pattern embeddings of standard parabolic subgroups into Weyl groups. 
	This allows us to embed the `seed' singularities of \autoref{S:Seed} into large Schubert varieties. 
	We then show that this can be done for almost every sufficiently large Schubert variety.    
	\\
	\par 
	Fix a connected reductive algebraic group $G$, and a choice of Borel subgroup $B$ and maximal torus $T$ such that $T \subseteq B \subseteq G$.  
	Let $R$ be the root system of $G$ arising from the choice of maximal torus $T$, and denote the positive (resp. negative) roots arising from $B$ by $R_+$ (resp. $R_-$). 
	We denote by $V$ the real inner product space spanned by the root system $R$. 
	When we wish to emphasise this data, along with the Weyl group, we write $(W,R, V)$. 
	\\
	\par 
	The following notions are from \cite[\S 2]{Woo-arbitrary}. 
	A subsystem embedding $\iota$ of $(W',R', V')$ into $(W,R, V)$ is an injective linear map $\iota: V' \hookrightarrow V$, such that $R' \cong \iota(V') \cap R$. 
	Such a map $\iota$ identifies $W'$ with the subgroup generated by the reflections corresponding to roots in $\iota(R')$. 
	Denote the inversion set $R_+ \cap xR_-$ of $x \in W$ by $\Inv(x)$, and recall that $x \in W$ is determined by $\Inv(x)$.
	Then define $\phi: W \rightarrow W'$ so that $\phi(x)$ is the unique element satisfying $\Inv(\phi(x)) = \iota^{-1} (\Inv(x) \cap \iota(R_+'))$. 
	Note that $\phi$ evidently depends on $\iota$, though we omit this from notation.
	An element $w \in W'$ is said to pattern embed in $x \in W$ if $\phi(x) = w$.
	Note that this generalises the usual type $A$ pattern embedding to arbitrary types. 
	Now fix $u,w \in W'$ with $u \leq w$ and $x,y  \in W$ with $y \leq x$. 
	Then $[u,w]$ is said to interval pattern embed into $[y,x]$ when the following three conditions are satisfied:
	(1) $\phi(x) = w$ and $\phi(y)=u$; 
	(2) $x$ and $y$ are in the same right $\iota(W')$-coset; and 
	(3) the intervals $[u,w]$ and $[y,x]$ are isomorphic as posets. 
	Note that $[u,w]$ and $x$ determine $y$, so we may speak of embedding an interval $[u,w]$ in $x$. 
	\\
	\par 
	Suppose that we have a seed giving rise to $p$-torsion in the (co)stalks of an intersection cohomology sheaf, as in \autoref{L:Seed Torsion}, or a non-perverse parity sheaf, as in \autoref{L:Seed Perverse}. 
	An interval pattern embedding produces other Schubert varieties whose intersection cohomology sheaf or parity sheaf exhibits the same behaviour. 
	
	\begin{Lemma}\label{L:pKL}
		Let $G$ be a classical algebraic group with Weyl group $W$ and $u,w \in \Sn_m$. 
		If $[u,w]$ interval pattern embeds in $[y,x]$ where $x,y \in W$, then ${}^p h_{u,w} = {}^p h_{y,x}$. 
	\end{Lemma}
	\begin{proof}
		Let $T' \subset B' \subset PGL_m$ be a choice of Borel subgroup and maximal torus inside $PGL_m$, and identify $\Sn_m$ with the Weyl group of $PGL_m$. 
		To invoke the arguments of \cite[\S7]{FW14} we will henceforth abuse notation and replace $G$ by its adjoint form; such a replacement is harmless as the isogeny induces an isomorphism of flag varieties that respects the Schubert stratification. 
		Further, let $T \subset B \subset G$ be a choice of maximal torus and Borel subgroup inside $G$, with root system $R$ and Weyl group $W$.
		We denote by $B_-'$ and $B_-$ the opposite Borel subgroups of $PGL_m$ and $G$ respectively. 
		The main result of \cite[\S3]{Woo-arbitrary} states that if $[u,w]$ interval pattern embeds in $[y,x]$ then the Richardson variety $X_{\,w}^{'u} = \overline{B'wB'/B'} \cap \overline{B'_-uB'/B'}$ is isomorphic to the Richardson variety $X_x^y =\overline{BxB/B} \cap \overline{B_-yB/B}$. 
		Moreover, the main result of \cite[\S4]{Woo-arbitrary} implies this isomorphism sends $T'$-fixed points of $X_{\,w}^{'u}$ to $T$-fixed points of $X_x^y$, and one-dimensional $T'$-orbits to one-dimensional $T$-orbits, and torus characters which appear as edge labels to the corresponding torus characters under the embedding $\iota$. 
		Consequently, the moment graph of $X_{\,w}^{'u}$ is isomorphic, as a directed, labelled subgraph, to a subgraph in the moment graph of $X_x^y$, with $wB'/B'$ mapping to $xB/B$. 
		The moment graphs of the Richardson varieties $X_{\,w}^{'u}$ and $X_x^y$ naturally embed in the moment graphs of the Schubert varieties $\overline{B'wB'/B'}$ and $\overline{BxB/B}$ respectively.  
		Consequently, the directed, labelled subgraph of the moment graph of $\overline{B'wB'/B'}$ with vertices $[u,w]$ is isomorphic to the directed, labelled subgraph of the moment graph of $\overline{BxB/B}$ with vertices $[y,x]$.  
		The main result of \cite{FW14} implies that, subject to various technical conditions, applying the Braden--MacPherson algorithm to the moment graph of each Schubert variety computes the stalks of the corresponding indecomposable parity sheaf. 
		The technical conditions are satisfied for Schubert varieties as we assume parity sheaves have coefficients in $\Z_p$ and $G$ is of adjoint type; see \cite[\S 7]{FW14}. 
		Since the algorithm is computed by downward induction on the Bruhat order, the previous isomorphism of directed, labelled subgraphs implies ${}^p h_{u,w} = {}^p h_{y,x}$. 
	\end{proof}

	\begin{Remark}\label{R:KLpoly}
		An analogous argument also shows that if $[u,w]$ interval pattern embeds in $[y,x]$ where $x,y \in W$, then the classical Kazhdan--Lusztig polynomials satisfy $h_{u,w} = h_{y,x}$.	
		In particular, one should use a parity sheaf with coefficients in a field of characteristic 0 and identify it with the corresponding intersection cohomology sheaf using \cite[\S2]{JMW-parity}. 
	\end{Remark}

	We will only consider embeddings that arise from the following construction.
	Fix $I \subseteq S$. 
	Set $W_I$ to be the group generated by $s \in I$.
	Recall that ${}^I W = \{ z \in W \, | \, wz \geq z \text{ for all } w \in W_I \}$ denotes the set of minimal length coset representatives.   
	
	\begin{Lemma}\label{L:IPE}
		Fix a standard parabolic subgroup $W_I \subseteq W$, a minimal length coset representative $z \in {}^I W$, and elements $u,w \in W_I$ with $u \leq  w$. 
		Then $[u,w]$ interval pattern embeds into $[uz, wz]$. 
	\end{Lemma}
	
	\begin{proof}
		We first  prove condition (1). 
		Since $W_I$ is a standard parabolic subgroup, its root system $R_I$ is naturally identified with the root subsystem $R_I \subseteq R$ where $\iota$ is the obvious embedding (henceforth we omit $\iota$ from notation).
		The condition $\phi(wz)=w$ is equivalent to showing $\Inv(w) = \Inv(wz) \cap (R_I)_+$. 
		Now note that $\Inv(wz) = \Inv(w) \cup w \Inv(z)$, since $\ell(wz) = \ell(w)+ \ell(z)$. 
		Further, $\Inv(z) \cap (R_I)_+ = \emptyset$, as $z \in {}^I W$. 
		Finally, since $w$ preserves $R_I$ and is bijective, it follows that $w \Inv(z) \cap (R_I)_+ = \emptyset$. 
		An identical argument shows $\phi(uz)=u$.
		Hence condition (1) is satisfied. 
		Condition $(2)$ holds by construction.
		Finally, condition $(3)$ follows from the Bruhat order preserving  property of minimal coset representatives and the fact $[uz,wz] \subseteq W_I z$; see \cite[\S2.5]{BB08}. 
	\end{proof}

	We can now prove the main results of this paper. 

    \begin{proof}[Proof of Theorems \ref{T:Main} and \ref{T:Main2}]
		To prove \autoref{T:Main}, set $m=T(p)$ and $w \in \Sn_m$ as in \autoref{L:Seed Torsion}; to prove \autoref{T:Main2}, set $m=P(p)$ and $w \in \Sn_m$ as in \autoref{L:Seed Perverse}. 
		Let $r$ denote the rank of a maximal, irreducible, type-$A$ standard parabolic subgroup in $W$.
		In particular, in type $A_n$ we have $r=n$, and in types $B_n$, $C_n$ and $D_n$ we have $r=n-1$. 
		Set $k = \lfloor (r+1)/m\rfloor$.
		Let $W_I \subseteq W$ be a standard parabolic subgroup of type $A_{m-1} \times \dots \times A_{m-1}$, with $k$ factors, i.e. $W_I \cong \Sn_{m} \times \dots \times \Sn_{m}$.
		For any $x \in W$ we can write $x=vz$, where $v = w_1 \dots w_k \in W_I$, with each $w_i \in \Sn_{m}$, and $z \in {}^I W$ is a minimal coset representative.
		Since the $w_i$ commute, we can unambiguously define $v_i = \prod_{j \neq i} w_j$, and $z_i = v_i z$.  
		Hence $x = w_i z_i$, and $z_i$ is a minimal coset representative for the $i$-th factor of $\Sn_m$ inside $W_I$, and in turn $W$.  
		If some $w_i = w$, then for any $u \leq w$,  \autoref{L:IPE} implies $[u,w]$ interval pattern embeds in $[uz_i, x]$.   
		Then the proportion of elements in $W$ that $[u,w]$ embeds in is at least
 		\begin{align*}
			\frac{|W| - (m! - 1)^k \, |{}^I W|}{|W|}
			=
			1-\left(1-\frac{1}{m!}\right)^k.
		\end{align*}
		Since $k$ grows linearly with $n$, the complementary proportion decays exponentially in $n$.
		Assuming there is some $i$ such that $w_i=w$, then, for any $u \leq w$, \autoref{L:pKL} implies ${}^p h_{u,w} = {}^p h_{uz_i , x}$ and similarly $ h_{u,w} =  h_{uz_i , x}$ by \autoref{R:KLpoly}.  
		 Again, downward induction using the unitriangularity of Kazhdan--Lusztig polynomials implies ${}^p a_{u,w} = {}^p a_{uz_i , x}$ for all $u \leq w$. 
		 In particular, if we now choose $u$ such as in \autoref{L:Seed Torsion} then \autoref{L: torsion equivalences} implies \autoref{T:Main}, and if we choose $u$ as in \autoref{L:Seed Perverse} then \autoref{L: numerical non-perversity} implies \autoref{T:Main2}. 
	\end{proof}

	\begin{Remark}
		In practice the complementary proportion appears to decay much faster than our bound in the proof of \autoref{T:Main} suggests; see \autoref{S:data}. 
	\end{Remark}

\section{Different species}\label{S:Partial}

	We now consider $p$-Kazhdan--Lusztig theory for certain families of partial flag varieties. 
	In this setting, there are two distinct classes of $p$-Kazhdan--Lusztig bases to consider: 
	spherical $p$-Kazhdan--Lusztig bases $\{ {}^p c_{x} \, | \, x \in {}^I W \}$, which encode the stalks of indecomposable parity sheaves; and 
	anti-spherical $p$-Kazhdan--Lusztig bases $\{ {}^p d_{x} \, | \, x \in {}^I W \}$ which encode the filtration multiplicities of indecomposable tilting sheaves; see \cite[\S2.6]{Williamson-hecke-category} or \cite[\S2]{Baine-cominuscule} for more details. 
	In particular, one can compute the stalks of mixed intersection cohomology sheaves from the anti-spherical $p$-Kazhdan--Lusztig basis using the arguments in \cite[\S 8]{Baine-cominuscule}. 
	We denote the classical spherical and anti-spherical Kazhdan--Lusztig bases by $\{ c_{x} \, | \, x \in {}^I W \}$ and $\{ d_{x} \, | \, x \in {}^I W \}$ respectively. 
	\\
	\par
	The asymptotic behaviour of $p$-Kazhdan--Lusztig bases differs greatly depending on both the family of partial flag varieties and the characteristic $p$.  
	For a partial flag variety $X_n$ we denote the corresponding minimal coset representatives indexing its Bruhat stratification by ${}^I W_n$. 
	
	\begin{Proposition}
		For each $L \in \left\lbrace 0, \tfrac{1}{2}, 1 \right\rbrace$, there exist families of partial flag varieties $\{X_n \}$ such that 
		\begin{align*}
			\lim_{n \rightarrow \infty}
			\frac{|\{ x \in {}^IW_n  ~|~ {}^2 d_x  \neq d_x \}|}{|{}^I W_n|}
			= L.
		\end{align*}
	\end{Proposition}
	
	\begin{Remark}
		In \cite[\S7.2]{Baine-cominuscule} two definitions of the set $E(x)$ are provided. 
		The first contains a typographic error and should read 
		\begin{align}
		\label{E: E Set}
			E(x) 
			=
			\{
			t \in x
			\, | \, 
			t \text{ is even and }
			|\{ t' \in x \, | \, t> t' \geq k \}| \geq  |\{ t' \notin x \, | \, t> t' \geq k \}|
			\text{ for all } t>k\geq 1
			\}.
		\end{align}
		This definition now agrees with the one provided in \cite[Equation (2)]{Baine-cominuscule}. 	
	\end{Remark}
	
	\begin{proof}
	It follows from \cite{BowmanDeVisscherHaziNorton-hermitian,Baine-cominuscule} that the limit is $0$ when $\{X_n\}$ is any family of minuscule flag varieties.
		When $\{ X_n\}$ is a family of odd-dimensional quadric hypersurfaces it follows from \cite[\S7.8]{Baine-cominuscule} that 
		\begin{align*}
			\frac{|\{ x \in {}^I W_n ~|~ {}^2 d_x  \neq d_x \}|}{|{}^I W_n|}
			=
			\frac{n-1}{2n}
			\longrightarrow
			\frac{1}{2}. 
		\end{align*}
		
		The most interesting case occurs when $\{ X_n \}$ is the family of Lagrangian Grassmannians (i.e. when $X_n = LG(n,2n)$). 
		In \cite[\S7.1]{Baine-cominuscule} it is explained that Schubert varieties in the Lagrangian Grassmannian $LG(n,2n)$ are naturally indexed by subsets $x \subseteq \{ 1, \dots, n \}$. 
		To each subset $x$ we can associate a set $E(x)$, as in \autoref{E: E Set}.
		It follows from \cite[\S1.3]{Baine-cominuscule} that ${}^2 d_x = d_x$ if and only if $E(x) = \emptyset$. 
		Hence it suffices to determine $|\{ x \subseteq \{ 1, \dots ,n \} \, | \, E(x) = \emptyset \}|$.
		First, fix $x \subseteq \{ 1 , \dots , n \}$ and define $\sigma_t = 1$ if $t \in x$, and $\sigma_t = -1$ if $t \notin x$. 
		If we set $\Sigma_0=0$ and $\Sigma_t = \sigma_1 + \dots + \sigma_t$, then \autoref{E: E Set} becomes
		\begin{align}
		\label{E: E Set new }
			E(x) 
			=
			\{
			t \in x
			\, | \, 
			t \text{ is even and }
			\Sigma_{t-1} \geq \Sigma_{k-1}
			\text{ for all } t>k\geq 1
			\}.
		\end{align}  
		We claim $E(x) = \emptyset$ if and only if $\Sigma_t < 2$ for all $t$. 
		Suppose $t \in E(x)$. 
		Then $\Sigma_t > \Sigma_k$ for all $0\leq k<t$ by \autoref{E: E Set new } and the fact $\sigma_t =1$. 
		Moreover, since $t$ is even, $\Sigma_t$ is even, and hence $\Sigma_t \geq 2$. 
		Conversely, suppose $t$ is the first time $\Sigma_t = 2$. 
		Note that $t$ is necessarily even and $\Sigma_{t-1}=1\geq \Sigma_{k-1}$ for all $t>k\geq1$, so $t\in E(x)$. 
		Hence $E(x) = \emptyset$ if and only if $\Sigma_t < 2$ for all $t$.
		Equivalently $|\{ x \subseteq \{ 1, \dots ,n \} \, | \, E(x) = \emptyset \}|= |\{ x \subseteq \{ 1, \dots ,n \} \, | \, \Sigma_{k} < 2  \text{ for all } k\}|$. 
		If we view $(t, \Sigma_t)$ as a walk in $\Z^{2}$ starting at $(0,0)$, then the number of walks with exactly $i$ up-steps and maximum height $<2$ is   
		\begin{align*}
			|\{ x \subseteq \{ 1, \dots ,n \} \, | \,  \Sigma_{k} < 2  \text{ for all } k \text{ and } |x| =i\}|
			=
			\binom{n}{i} - \binom{n}{i-2}.
		\end{align*}
		Note that any $x \subseteq \{ 1, \dots , n \}$ satisfying $\Sigma_k <2$ for all $k$ must satisfy $|x| \leq  \lfloor \frac{n+1}{2} \rfloor$.  
		Consequently 
		\begin{align*}
			\frac{|\{ x \subseteq \{ 1, \dots ,n \} \, | \, E(x) = \emptyset \}|}{|{}^I W_n|}
			= 
			\frac{1}{2^{n}}\left(\binom{n}{\lfloor \frac{n+1}{2} \rfloor} + \binom{n}{\lfloor \frac{n+1}{2} \rfloor-1} \right)
			=
			\frac{1}{2^{n}}\binom{n+1}{\lfloor \frac{n+1}{2} \rfloor}
			\sim 
			2\sqrt{\frac{2}{\pi (n+1)}},
		\end{align*}
		where the asymptotic formula follows from Stirling's approximation. Hence
		\[
			\frac{|\{ x \in {}^IW_n ~|~ {}^2 d_x \neq d_x \}|}{|{}^I W_n|}
			=
			1-2\sqrt{\frac{2}{\pi (n+1)}}
			+ o\left(n^{-1/2}\right)
			\longrightarrow 1,
		\]
	where $o\left(n^{-1/2}\right)$ denotes little-o notation. 
    \end{proof}
    
    \begin{Remark}
    	Note that the results of \cite{Baine-cominuscule} imply that for all of the aforementioned partial flag varieties, 	the corresponding ratios are
    	\begin{align*}
			\frac{|\{ x \in {}^IW_n  ~|~ {}^p d_x  \neq d_x \}|}{|{}^I W_n|}
			=0,
			&&
			\frac{|\{ x \in {}^IW_n  ~|~ {}^p c_x  \neq c_x \}|}{|{}^I W_n|}
			=0
		\end{align*}
		for all $n$, when $p$ is a good prime for the corresponding classical group $G$. 
    \end{Remark}

\section{Surveying the landscape}\label{S:data}
    
    The $p$-Kazhdan--Lusztig bases can be computed for Weyl groups in small rank using the ASLoc package of Gibson--Jensen--Williamson \cite{ASLoc}. 
    These examples indicate that the proportion of $x \in W$ satisfying ${}^p b_x \neq b_x$ grows \emph{much} faster than the growth suggested by the proof of \autoref{T:Main} and \autoref{T:Main2}. 
    \autoref{F:p2-percentage} displays this growth when $p=2$ for Weyl groups of various types and ranks. 
    (We have omitted type $A_n$, with $n\leq 6$, from the figure, since one always has ${}^p b_x=b_x$, for all $p$, in these groups.)
    Further computational data are available in \cite{BT}.
   The authors think it is an interesting question to determine more precise asymptotic formulas for the growth of $p$-torsion in the (co)stalks of intersection cohomology sheaves on Schubert varieties.
   
    \begin{figure}[ht]
		\centering
		\includegraphics[width=0.75\textwidth]{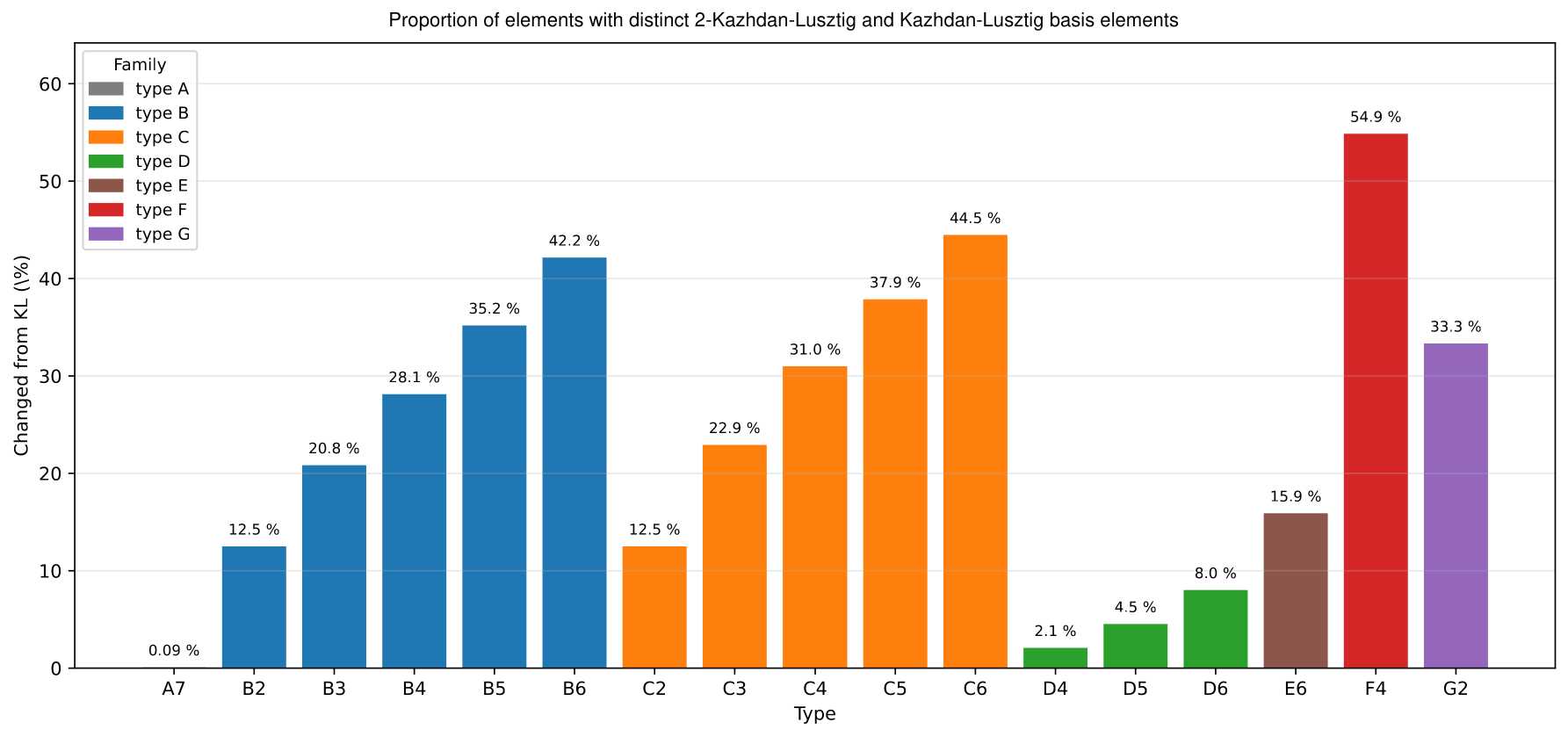}
		\caption{The proportion of elements $x \in W$ for which ${}^2 b_x \neq b_x$.}
		\label{F:p2-percentage}
	\end{figure}

\end{document}